\documentclass[12pt]{amsart}
\usepackage{mathrsfs}
\usepackage{latexsym,amsthm,amssymb,amsmath,amsfonts,graphicx,epstopdf,float}
\usepackage{bbm}
\usepackage{pifont}
\usepackage{cite}
\usepackage{multirow}
\usepackage{appendix}
\usepackage{tikz,xcolor}
\usepackage{graphicx}
\usepackage{subcaption}
\usepackage{setspace}

\usetikzlibrary{decorations.markings}
\usetikzlibrary{intersections}
\usetikzlibrary{calc}

\newtheorem{theorem}{Theorem}[section]
\newtheorem{thm}[theorem]{Theorem}
\newtheorem{prop}{Proposition}[section]

\newtheorem{lemma}{Lemma}[section]
\newtheorem{remark}{Remark}[section]

\numberwithin{equation}{section}

\newcommand{\bp}{{\mathbf{p}}}

\newcommand{\SD}{{\mathcal D}}

 \usepackage{color}

\begin{document}

\title{Doubling  property of self-similar measures with overlaps  }

\author{Yu Wang}
\address{Department of Mathematics and Statistics, Central China Normal University, Wuhan, 430079, China}
\email{wy2022@mails.ccnu.edu.cn}

\author{Ya-min Yang $^*$}
\address{Institute of Applied Mathematics, College of Informatics, Huazhong Agricultural University, Wuhan, 430070, China}
\email{yangym09@mail.hzau.edu.cn}

 \thanks {The work is supported by NSFC No.12071167}

 \thanks{{\bf 2010 Mathematics Subject Classification:}  28A80, 26A16\\
 {\indent\bf Key words and phrases:}\ doubling property, self-similar measure, Pisot number}

\thanks{* Corresponding author.}
\begin{abstract}  Recently, Yang, Yuan and Zhang [{\it Doubling properties of self-similar measures and Bernoulli measures on self-affine Sierpi¨½ski sponges,} Indiana Univ. Math. J., \textbf{73} (2024), 475¨C492] characterized when a self-similar measure satisfying the open set condition is doubling. In this paper, we study when   a self-similar measure with overlaps is doubling.

Let $m\geq 2$ and let  $\beta>1$ be the Pisot number satisfying $\beta^m=\sum_{j=0}^{m-1}\beta^j$.
Let  $\mathbf{p}=(p_1,p_2)$ be a probability weight and let $\mu_\bp$ be the self-similar measure associated to
the IFS $\{ S_1(x)={x}/{\beta},
S_2(x)={x}/{\beta}+(1-{1}/{\beta}),\}.$ Yung [...,Indiana Univ. Math. J., ] proved that when $m=2$, $\mu_\bp$ is doubling if and
only if $\bp=(1/2,1/2)$. We show that for $m\geq 3$, $\mu_\bp$ is always non-doubling.
\end{abstract}
\maketitle


 \section{\textbf{Introduction}}

   Let  $1<\beta<2$  be a real number  and set
  \begin{equation}\label{eq:IFS}
  S_1(x)=\frac{x}{\beta}, \quad
S_2(x)=\frac{x}{\beta}+(1-\frac{1}{\beta}),
\end{equation}
then $\{S_1,S_2\}$ is an iterated function system  (IFS) and its  attractor is $[0,1]$. Let $\mathbf{p}=(p_1,p_2)$ be a probability weight, \textit{i.e.}, $p_1,p_2>0$ and $p_1+p_2=1$. Let
\begin{equation}\label{eq:mu}
\mu_\mathbf{p}=p_1\mu_\mathbf{p}\circ S_1^{-1}+p_2\mu_\mathbf{p}\circ S_2^{-1}
\end{equation}
be the self-similar measure on $[0,1]$  associated with IFS $\{S_1, S_2\}$, see Hutchinson \cite{Hutchinson_1981}.
When $\bp=(1/2,1/2)$, $\mu_\bp$ is the famous Bernoulli convolution measure.

There are a lot of works devoted to determine when $\mu_\bp$ is absolutely continuous,
 for example, Erd\"{o}s \cite{Erdos_1939}, Garsia \cite{Garsia_1962},
 Solomyak \cite{Solomyak_1995}, \textit{etc.}.
In this paper, we investigate the doubling property of $\mu_\bp$.

 Doubling measure was first introduced by Beurling and Ahlfors\cite{Beurling_1956} and has become a fundamental concept in analysis on metric space. Recall that a Borel probability measure $\mu$ on a metric space is called a \emph{doubling} measure if
$$\underset{x\in \text{supp}(\mu),~r>0}{\sup}~\frac{\mu(B(x,2r))}{\mu(B(x,r))}<\infty,$$
where $\text{supp}(\mu)$ means the support of $\mu$.  Recently, there
are several works focus on studying when a Bernoulli measure of an IFS is doubling. Olsen \cite{Olsen_1995} proved that under the strong separation condition, every self-similar measure is doubling. Mauldin and Urba\'nski \cite{Mauldin_1996} proved that under the open set condition, the canonical self-similar measure is doubling. Yung \cite{Yung07} provided necessary and sufficient conditions for self-similar measures to be doubling. Employing  a novel technique
 called \emph{weighted neighbor graph},  Yang, Yuan and Zhang \cite{Yang_2024} completely characterize when  a self-similar measure satisfying the open set condition  is doubling, or when a Bernoulli measure on a self-affine Sierpinski sponge is doubling.

 Let $m\geq 2$ be an integer and let $\beta>1$ be a root of the polynomial
\begin{equation}\label{eq:Pisot}
P(x)=x^m-\sum_{j=0}^{m-1} x^j ,
\end{equation}
 then $\beta$ is a Pisot number (\cite{Bertin92}). When $m=2$, $\beta$ is the golden number.
Yung proved that

\begin{prop}\label{prop:Yung} If $m=2$, then  $\mu_{\bp}$ is doubling if and only if $p_1=p_2=1/2$.
\end{prop}

In this paper, we show that

\begin{thm}\label{thm:main} For $m\geq 3$, $\mu_\mathbf{p}$ is always  non-doubling.
\end{thm}

It is somehow surprising that for $m\geq 3$, even for $\bp=(1/2,1/2)$, $\mu_\bp$ is not doubling.

\begin{remark} \emph{(i) If $\beta\in (1, 2)$ is a Pisot number, the method in this paper actually gives us an algorithm to determine where $\mu_{\bp}$ is doubling.  However, we boldly conjecture that if $1<\beta<2$ is a Pisot number other than the golden ratio, then $\mu_{\bp}$ is always non-doubling.}

\emph{(ii) If $\beta$ is not a Pisot number, then we even do not have an algorithm to check whether $\mu_{\bp}$ is doubling. }
\end{remark}

The paper is organized as follows. In Section 2, we introduce a  substitution related to the Pisot number $\beta$,
and obtain a graph-directed structure  of  the interval $[0,1]$. In Section 3, we introduce the weight function and obtain
 a new criterion of doubling property of $\mu_{\bp}$.    We prove Theorem 1.1 in Section 4. In Section 5, we give an alternative proof of Proposition \ref{prop:Yung}, a result of Yung.

\section{\textbf{A graph-directed structure of $[0,1]$}}

We fix $m\geq 2$, and let $\beta$ be the Pisot number determined by \eqref{eq:Pisot}. For simplicity, we will denote $$0.a_1\dots a_n=\sum_{j=1}^n a_j/\beta^j.$$
 Recall that $\{S_1,S_2\}$ is the IFS defined by \eqref{eq:IFS}. We denote $\Sigma=\{1,2\}$, and let $\Sigma^k=\{i_1\dots i_k;~ i_j\in \Sigma\}$ and $\Sigma^*=\bigcup_{k=0}^\infty\Sigma^k$. For $I=i_1\cdots i_k$, we call it a word of length $k$, and we use $|I|$ to denote the length of  $I$.

First, we define basic intervals as follows.
  Set $X_0=\{0,1\}$ and
\begin{equation}
X_n=\{S_I(0);|I|=n\}\cup \{1\}.
\end{equation}
Let us arrange the points in $X_n$ in ascending order as
\begin{equation}
 0=a_{n,0}<a_{n,1}<\cdots<a_{n,\#X_n-1}=1,
\end{equation}
where $\#X_n$ denotes the cardinality of $X_n$.  We call the interval $[a_{n,j}, a_{n,j+1}]$  a  \emph{basic interval} of  rank $n$.

Next, we define a sequence $(d_j)_{j=1}^m$ as follows.
Set  $d_0=1$, $d_1=\beta-1=0.1^{m-1}$ (in base $\beta$), and set
\begin{equation}\label{eq:D}
d_{j+1}=\beta d_j-d_1, \quad \text{ for }  1\leq j\leq m-1.
\end{equation}
Then
$$d_j=0.\underbrace{1\cdots1}_{m-j}
0\underbrace{1\cdots1}_{j-1}, \quad \text{ for }  2\leq j\leq m.$$
 It is seen that $d_0>d_1>\cdots>d_m$.
Denote
\begin{equation}
\mathcal{D}=\{d_0,d_1,
\cdots d_m\}.
\end{equation}
Since $\beta d_0-d_1=d_0$ and $\beta d_m-d_1=0$, we have
\begin{equation}\label{eq:D_0}
\beta \SD-d_1\subset \SD\cup\{0\}.
\end{equation}

 \indent For $I,J\in \Sigma^*,$ we use $I\ast J$ to denote the concatenation of them.
\begin{lemma} The distance between two adjacent points in $X_n$ belongs to
\begin{equation}
 \beta^{-n}\mathcal{D}.
\end{equation}

\end{lemma}
\begin{proof}
We prove the lemma by induction on $n$.

 If $n=1$, then $S_2(0)-S_1(0)=\beta^{-1}\cdot d_1$, $1-S_2(0)=\beta^{-1}\cdot d_0$,
  and the lemma holds.
\indent Suppose the lemma holds for $n=k(k\geq 1)$. We pick $a_{k,j}\in X_k$, $j\neq \#X_k-1$ and suppose $a_{k,j}=S_I(0)$ for some $I\in \Sigma^k$, then $S_{I*1}(0)=a_{k,j}$ and $$S_{I*2}(0)=S_{I}\circ S_2(0)=
\beta^{-(k+1)}\cdot d_1+a_{k,j}.$$
Since $S_{I*2}(0)-S_{I*1}(0)=\beta^{-k}\cdot d_m  = \min \beta^{-k}\mathcal{D}$, so $a_{k,j}<S_{I*2}(0)\leq a_{k,j+1}$. Then
\begin{equation}
X_{k+1}=\{a_{k,j},a_{k,j}+\beta^{-(k+1)}\cdot d_1;a_{k,j}\in X_k\}\cup \{1\}.
\end{equation}
We see that the distance between two adjacent points in $X_{k+1}$ is either $\beta^{-(k+1)}\cdot d_1$ or $a_{k,j+1}-(a_{k,j}+\beta^{-(k+1)}\cdot d_1)$.
The later one belongs to $\beta^{-(k+1)}(\beta \mathcal{D}-d_1) $  by inductive hypothesis,  and hence belongs to $\beta^{-(k+1)}\mathcal{D}$ by  \eqref{eq:D_0}.  The lemma is proved.
\end{proof}

Next, we define a word  $T_n=t_{n,1}\cdots t_{n,\#X_n-1}$  by
\begin{equation*}
  t_{n,j}= \beta^n(a_{n,j}-a_{n,j-1}) \quad \text{for} \quad 1\leq j\leq \#X_n-1.
\end{equation*}
By Lemma 2.1, $T_n$ is a word over $\SD$.
For convenience,  we define $t(a_{n,j-1})=t_{n,j}$, and call $t_{n,j}$ be the label of $a_{n,j-1}$ for $1\leq j\leq \#X_n-1$.

Let $\sigma$ be a substitution over $\SD$ defined by:
$$ \sigma: \left \{
\begin{array}{l}
  d_0\mapsto d_1d_0,\\
   d_i\mapsto d_1d_{i+1},   \text{ for } 1\leq i\leq m-1, \\
    d_m\mapsto d_1.
    \end{array}
    \right .
$$

\begin{remark}\emph{We remark that if we restrict $\sigma$ to the letters $d_1,\dots, d_m$, then $\sigma$ coincides with
the substitution corresponding to the $\beta$-numeration system. }
\end{remark}

\begin{thm}
 $T_n=\sigma^n(d_0)$   for all $n\geq 1$.
\end{thm}

\begin{proof} If $n=1$, then $X_1=\{0, d_1/\beta, 1\}$, so $T_1=d_1d_0=\sigma(d_0)$.

To prove the lemma, we only need to show that $T_{k+1}=\sigma(T_k)$ for $k\geq 1$.
Pick $1\leq j\leq |X_k|-1$ and denote $t_{k,j}=d_i$, then the length of the interval $[a_{k,j-1},a_{k,j}]$ is $\beta^{-k} d_i$.
 Since
 $$X_{k+1}=X_k\cup \left (X_k+\frac{d_1}{\beta^{(k+1)}}\right ),$$
 in $X_{k+1}$,
 there is at most one point between $a_{k,j-1}$ and $a_{k,j}$, and this point is
 $$c=a_{k,j-1}+\beta^{-(k+1)}  d_1.$$
     If $1\leq i\leq m-1$,  then $a_{k,j-1}<c<a_{k,j}$, so $c$ breaks $[a_{k,j-1}, a_{k,j}]$ into two subintervals, and an easy calculation shows that the label of these subintervals are   $d_1$ and  $d_{i+1}$.
 In terms of substitution, we have $d_i\mapsto d_1d_{i+1}$.

   If $i=0$, similar as above,  we have   $d_0\mapsto d_1d_0$.

    If $i=m$, then $c=a_{k,j}$, and we have $d_m\mapsto d_1$.

    Hence we have $T_{k+1}=\sigma(T_{k})$, the lemma is proved.

\end{proof}

\begin{lemma}  Let $I\in \Sigma^n(n\geq 1)$. If $S_I(0)=a_{n,j}$ and  $j\leq\#X_n-3$, then
$$[a_{n,j},a_{n,j+1}]\subset S_I([0,1]) \subseteq [a_{n,j},a_{n,j+2}].$$
Consequenctly, the interior of $S_I([0,1])$ intersects exactly two basic intervals of rank $n$.
\end{lemma}

\begin{proof}
Clearly   $d_1+d_i>d_1+d_m>1$ for every  $1\leq i\leq m-1.$
 From the definition of the substitution $\sigma$, we see that  at least one of letter of $t_{n,j} t_{n, j+1}$ is $d_1$.
 Let $d_i$ be  the label of the other interval, then
 $$|a_{n,j+2}-a_{n,j}|\geq (d_1+d_i)/\beta^n>1/\beta^n=|S_I([0,1])|.$$
 The lemma is proved.
%
\end{proof}

\section{\textbf{Measure of basic  intervals}}

  For $I=i_1\cdots i_k\in \Sigma^{*}$, we define $p_I=p_{i_1\cdots i_k}=p_{i_1}\cdots p_{i_k}$ and $S_I=S_{i_1\cdots i_k}=S_{i_1}\cdots S_{i_k}$. Iterating \eqref{eq:mu} $n$ times, we obtain
\begin{equation}\label{eq:mu_n}
\mu_\mathbf{p}=\sum_{I\in \Sigma^n} p_I\mu_\mathbf{p}\circ S_{I}^{-1}.
\end{equation}

 \begin{lemma} For $1\leq i\leq m$, we have
$$ \mu_{\bp}([0,d_i])=1-\frac{p_2^{m+1-i}}{1-p_1p_2^{m-1}}.$$
 \end{lemma}
 \begin{proof} Let us denote $c_{i}=\mu_\mathbf{p}([0,d_i])$, $1\leq i\leq m.$ By the definition of $\mu_{\bp}$, we have
 $$
 \mu_\mathbf{p}([0,d_i])=p_1\mu_\mathbf{p}([0,\beta d_i])+p_2\mu_\mathbf{p}([1-\beta, 1-\beta+\beta d_i]).
 $$
  If $1\leq i\leq m-1$, then $\beta d_i \geq 1$, so
$\mu_\mathbf{p}([0,d_i])=p_1+p_2\mu_\mathbf{p}([0,d_{i+1}])$, which implies that
\begin{equation}\label{eq:linear-1}
1-c_i=p_2(1-c_{i+1}).
\end{equation}
Applying  the same argument to $[0,d_m]$, we have
\begin{equation}\label{eq:linear-2}
c_m=\mu_\mathbf{p}([0,d_m])
=p_1\mu_\mathbf{p}([0,d_1])+
p_2\mu_\mathbf{p}(\{0\})=p_1c_1.
\end{equation}
 Solving the linear equation system   \eqref{eq:linear-1} and \eqref{eq:linear-2}, we obtain the lemma.
\end{proof}

We define a weight function $W_n(x)$ on  $X_n\setminus\{1\}$ as follows:
\begin{equation}
W_n(x)=\sum_{S_I(0)=x,\ |I|=n} p_I.
\end{equation}

  Fix $n\geq 2$, let $(x_0, x_1, x_2)=(a_{n,j},a_{n,j+1},a_{n,j+2})$, we call $(x_0, x_1, x_2)$ a \emph{triple} in $X_n$, 
and $(W_n(x_0), W_n(x_1), W_n(x_2))$ a \emph{weight triple} of rank $n$.

\begin{lemma}\label{lem:bound} Let $(x_0,x_1,x_2)$ be a triple in  $X_n$. Then there exist a constants $C'>0$ such that
\begin{equation}\label{eq:bound}
C'(W_n(x_0)+W_n(x_1))\leq\mu_\mathbf{p}([x_1, x_2])\leq W_n(x_0)+W_n(x_1).
\end{equation}
\end{lemma}

\begin{proof} Suppose $x_1-x_0=d_s/\beta^{n}$ and $x_2-x_1=d_t/\beta^n$   for some $0\leq s,t\leq m$. By \eqref{eq:mu_n}, we have

\begin{equation*}
\begin{aligned}
 \mu_\mathbf{p}([x_1, x_2]) &=\sum_{|I|=n} p_I  \mu_\bp(S_I^{-1}([x_1,x_2]))\\
&= \sum_{|I|=n, S_I(0)=x_0}
 p_{I}  \mu_\bp(S_I^{-1}([x_1,x_2]))+
  \sum_{|I|=n, S_I(0)=x_1}
  p_{I} \mu_\bp( S_I^{-1}([x_1,x_2]))\\
 \end{aligned}
\end{equation*}

 If $S_I(0)=x_0$,  then $ S_{I}^{-1}([x_1,x_2])\cap [0,1]=  [d_s,1]$;
if $S_I(0)=x_1$, then $ S_{I}^{-1}([x_1,x_2])\cap [0,1]=  [0, d_t]$.
Therefore,
$$\mu_\mathbf{p}([x_1, x_2])=W_n(x_0) (1-\mu_\bp([0,d_s]) +
  W_n(x_1) \mu_\bp([0, d_t]).
  $$
Set $C'=\min\{\mu_\bp([0, d_m]),1-\mu_\bp([0, d_t])\},$ we obtain \eqref{eq:bound}. The lemma is proved.
\end{proof}

\section{\textbf{Proof of the Theorem \ref{thm:main}}}

We use all the notations in the previous sections.  Let $B(z,r)$ be the ball of radius $r$ and with center  $z$.

\begin{lemma}\label{lem:interval}
The measure $\mu_\mathbf{p}$ is doubling on $[0,1]$ if and only if there exist constants $C_1'>0$ such that, for any $n\geq2$, $0\leq j\leq \#X_n-3$,
\begin{equation}\label{bound2}
(C_1')^{-1}\leq \frac{\mu_\mathbf{p}([a_{n,j},a_{n,j+1}])}{
\mu_\mathbf{p}([a_{n,j+1},a_{n,j+2}])}\leq C_1'.
\end{equation}
\end{lemma}

\begin{proof}
\indent The 'if' part.
Let $0<r<1$ and $z\in [0,1]$. Let $s$ be the  positive integer such that $1/\beta^{s}\leq r< 1/\beta^{s-1}$.
Let $j$ be the integer such that $z\in [a_{s,j},a_{s,j+1}]$.
On one hand, $B(z,r)\supset [a_{s,j},a_{s,j+1}]$ since $r\geq 1/\beta^s$, so
  $$\mu_\mathbf{p}(B(z,r))\geq\mu_\mathbf{p}([a_{s,j},a_{s,j+1}]).$$
   On the other hand,  $a_{s,k+1}-a_{s,k}=d_i/\beta^s\geq d_m/\beta^s=(\beta-1)/\beta^{s+1}$. Since $r<1/\beta^{s-1}$, then
   $$\frac{r}{(\beta-1)/\beta^{s+1}}<\frac{1/\beta^{s-1}}{(\beta-1)/\beta^{s+1}}=\frac{\beta^2}{\beta-1}\leq \beta^3<8.$$
 Therefore $B(z,2r)$ is covered by at most $16$ consecutive
  basic intervals of rank $s$.
    That is, $B(z,2r)\subset \bigcup_{k=j-8}^{j+8}[a_{s,k},a_{s,k+1}]$. By (\ref{bound2}), set $C_0=2\sum_{n=0}^8 (C_1')^n$,  we obtain that
    $$\mu_\mathbf{p}(B(z,2r))\leq C_0\mu_\mathbf{p}([a_{s,j},a_{s,j+1}])\leq C_0\mu_\mathbf{p}(B(z,r)),$$
    which confirms that  $\mu_\mathbf{p}$ is doubling.

\indent The 'only if' part. Suppose $\mu_\mathbf{p}$ is doubling, then there exists $C_0>0$ such that $\mu_\mathbf{p}(B(z,2r))\leq C_0\mu_\mathbf{p}(B(z,r))<\infty$ for any $z\in [0,1],r>0$.

Suppose on the contrary the result is false. Then  for any $M>0$, there exist $s\geq 2$ and $0\leq j\leq \#X_n-3$ such that
$$
\frac{\mu_\mathbf{p}([a_{s,j},a_{s,j+1}])}
{\mu_\mathbf{p}([a_{s,j+1},a_{s,j+2}])}>M$$
(or the other round).
 Choose $z$ and $r$ such that
  $\overline{B}(z,r)=[a_{s,j+1},a_{s,j+2}]$,   then $r=\frac{|a_{s,j+2}-a_{s,j+1}|}{2}\geq \frac{1}{2}\frac{d_m}{\beta^s}$. Therefore $|a_{s,j+1}-a_{s,j}|\leq \frac{d_0}{\beta^s}=\frac{1}{\beta^s}\leq \frac{2r}{d_m}$. Set $p=\lfloor 2/d_m\rfloor +1$, then
  $ [a_{s,j},a_{s,j+1}]\subset B(z, pr)$, so
$$\mu_\mathbf{p}(B(z,pr))\geq M \mu_\mathbf{p}(B(z,r)).$$
On the other hand, we have
$\mu_\mathbf{p}(B(z,pr))\leq \mu_\mathbf{p}(B(z,2^pr))\leq  C_0^p \mu_\mathbf{p} (B(z,r))$.
By choosing $M>C_0^p$, we obtain  a contradiction.
\end{proof}

\begin{lemma}\label{lem:balance-1} The measure $\mu_\mathbf{p}$  is doubling on $[0,1]$ if and only if  there exist constants $C_1,C_2>0$  such that for any $n\geq2$, $0\leq j\leq \#X_n-4$,
\begin{equation}\label{bound1}
C_1\leq \frac{W_n(a_{n,j})+W_n(a_{n,j+1})}
{W_n(a_{n,j+1})+W_n(a_{n,j+2})}
\leq C_2.
\end{equation}
\end{lemma}

\begin{proof}
Let $C_1=C'(C_1')^{-1}$ and $C_2=(C')^{-1}C_1'$, by \eqref{eq:bound}, we see that \eqref{bound2}  holds if and only if  (\ref{bound1}) holds. The lemma is proved.
\end{proof}

\begin{lemma}\label{lem:matrix}
Let $(x_0,x_1,x_2)$ be a triple in $X_n$,   and $w_1,w_2$ be the labels of $x_0$ and $x_1$ respectively. Let $(x_0, z_1, z_2, z_3)$
be four consecutive points in $X_{n+1}$. If $w_1\neq d_m$ and $w_2\neq d_m$, then
\begin{equation}\label{matrix1}
    \begin{bmatrix}
        W_{n+1}(z_{1})  \\
        W_{n+1}(z_{2})  \\
        W_{n+1}(z_{3})
    \end{bmatrix}
= M_1
    \begin{bmatrix}
        W_{n}(x_{0})  \\
        W_{n}(x_{1})  \\
        W_{n}(x_{2})
    \end{bmatrix}
\text{ where }
    M_1=\begin{bmatrix}
        p_2 & 0 & 0 \\
        0 & p_1 & 0 \\
        0 & p_2 & 0
    \end{bmatrix}.
\end{equation}
If $w_1w_2=d_md_1$, then \begin{equation}\label{matrix2}
    \begin{bmatrix}
        W_{n+1}(z_{1})  \\
        W_{n+1}(z_{2})  \\
        W_{n+1}(z_{3})
    \end{bmatrix}
= M_2
    \begin{bmatrix}
        W_{n}(x_{0})  \\
        W_{n}(x_{1})  \\
        W_{n}(x_{2})
    \end{bmatrix}
\text{ where }
    M_2=\begin{bmatrix}
        p_2 & p_1 & 0 \\
        0 & p_2 & 0 \\
        0 & 0 & p_1
    \end{bmatrix}.
    \end{equation}
\end{lemma}

\begin{proof}
 Suppose  $w_1\neq d_m$ and $w_2\neq d_m$, then
 $$z_1=x_0+\beta^{-(n+1)}\cdot d_1,  \ z_2=x_1,  \ z_3=x_1+\beta^{-(n+1)}\cdot d_1.$$
 So $z_1$ is a new partition point, or $z_1\not\in X_n$.
 Moreover, if $x_0=S_I(0)$ for some $I\in \Sigma ^{n}$, then $z_1=S_{I*2}(0)$, and for any other $J\in \Sigma^{n+1}$,
 $S_J(0)\neq z_1$. Therefore  we have
$$
W_{n+1}(z_1)=
p_2\sum_{S_I(0)=x_0,\ |I|=n} p_I =p_2W_{n}(x_0).
$$
  Similarly, we have $W_{n+1}(z_2)=p_1W_{n}(x_1)$, $W_{n+1}(z_3)=p_2W_{n}(x_1)$. These three equations imply \eqref{matrix1}.

 Suppose $w_1w_2= d_md_1$. Then
 $$z_1=x_1=x_0+\beta^{-(n+1)}\cdot d_1, \  z_2=x_1+\beta^{-(n+1)}\cdot d_1,  \  z_3=x_2.$$
 Notice that $z_1=S_J(0)$ for $J\in \Sigma^{n+1}$ if  and only if
 either $J=I*2$ where $S_I(0)=x_0$ or $J=I'*1$ where $S_{I'}(0)=x_1$.
 It follows that $W_{n+1}(z_1)=p_2W_{n}(x_0)+p_1W_n(x_1)$.
 Clearly, $W_{n+1}(z_2)=p_2W_{n}(x_1)$, $W_{n+1}(z_3)=p_1W_{n}(x_2)$.
 These three equations imply  \eqref{matrix2}.
\end{proof}


\begin{proof}[\textbf{Proof of Theorem \ref{thm:main}}] Let $\bp=(p_1,p_2)$. Set $\bp'=(p_2,p_1)$. Then it is easy to show that
$$\mu_{\bp}(A)=\mu_{\bp'}(1-A).$$
 It follows that  $\mu_\bp$ is doubling if and only if $\mu_{\bp'}$
is doubling. Therefore,   we may assume that $p_1\leq p_2$ without loss of generality.

   Set $I_n=112^{n-2}$ for $n\geq 2$. Let $z_{n,1}=S_{I_n}(0)$ and $(z_{n,1},z_{n,2},z_{n,3})$ be the triple in $X_n$ initialled by $z_{n,1}$.  It is seen that when $n=2$, we have
   $$
   (z_{2,1},z_{2,2},z_{2,3})=\frac{\left(0,d_1, d_1+d_2 \right)}{\beta^2}, \quad t(z_{2,1})t(z_{2,2})t(z_{2,3})=d_1d_2d_1,
   $$
    and  the associated weight triple  is $(p_1^2,p_1p_2,p_1p_2)$.

Now let $n\geq 3$. We study the word $t(z_{n,1})t(z_{n,2})t(z_{n,3})$. It is obvious that
\begin{equation}\label{eq:skip}
 z_{n,1}=z_{n-1,1}+d_1/\beta^{n}.
 \end{equation}
Write
 $$\sigma(t(z_{n-1,1})t(z_{n-1,2})t(z_{n-1,3}) )=u_1\dots u_k.$$
    By \eqref{eq:skip}, we see that $z_{n-1,1}$ and $z_{n,1}$ are two consecutive points in $X_n$, therefore
    $$t(z_{n,1})t(z_{n,2})t(z_{n,3})=u_2u_3u_4.$$

Let $n'$ be the integer in $\{2,\dots, m+1\}$ such that $n'\equiv n \pmod m$.  By the above argument,  we obtain that
\begin{equation}\label{eq:word}
t(z_{n,1})t(z_{n,2})t(z_{n,3})
=\left\{
  \begin{aligned}
  d_1d_2d_1&,\ \text{if} \ n'= 2,\\
    d_2d_1d_3&,\ \text{if} \  n'= 3, \\
d_{n'-1}d_1d_2&,\ \text{otherwise.}
  \end{aligned}
  \right..
\end{equation}
Figure 1 illustrates the above construction in the case $m=3$.
\begin{figure}[H]
  \centering
    \includegraphics[width=10 cm]{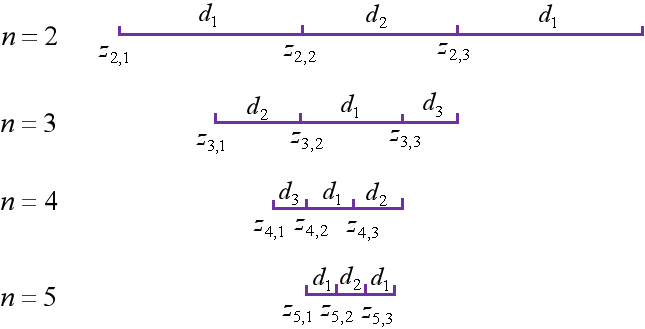}
    \caption{The case of $m=3$}
\end{figure}

By  \eqref{eq:word} and Lemma \ref{lem:matrix}, we have
\begin{equation*}
    \begin{bmatrix}
        W_{n+1}(z_{n+1,1})  \\
        W_{n+1}(z_{n+1,2})  \\
        W_{n+1}(z_{n+1,3})
    \end{bmatrix}
= Q_n
    \begin{bmatrix}
        W_{n}(z_{n,1})  \\
        W_{n}(z_{n,2})  \\
        W_{n}(z_{n,3})
    \end{bmatrix}
\end{equation*}
where $Q_n=M_1$ if $m\nmid (n-1)$ and $Q_n=M_2$ if $m|(n-1)$.
It follows that
\begin{equation}
    \begin{bmatrix}
        W_{km+2}(z_{km+2,1})  \\
        W_{km+2}(z_{km+2,2})  \\
        W_{km+2}(z_{km+2,3})
    \end{bmatrix}
 = (M_2M_1^{m-1})^k
    \begin{bmatrix}
        W_{2}(z_{2,1})  \\
        W_{2}(z_{2,2})  \\
        W_{2}(z_{2,3})
    \end{bmatrix}
    \end{equation}
    An easy calculation shows that
    \begin{equation}
    \left(M_2M_1^{m-1}\right)^k=
  \begin{bmatrix}
        p_2^{m} & p_1^m & 0 \\
        0 & p_1^{m-1}p_2 & 0 \\
        0 & p_1^{m-1}p_2 & 0
    \end{bmatrix}^k
    =
    \begin{bmatrix}
        p_2^{km} & \sum_{i=0}^{k-1}p_1^{m+(m-1)i}
        p_2^{km-m-(m-1)i} & 0 \\
        0 & p_1^{km-k}p_2^{k} & 0 \\
        0 & p_1^{km-k}p_2^{k} & 0
    \end{bmatrix},
\end{equation}
 so the weight triple of  $(z_{km+2,1},z_{km+2,2},z_{km+2,3})$  is
$$
 \left(p_1^2p_2^{km}+p_1^{m+1}p_2^{km-m+1}
 \left(\sum_{i=0}^{k-1} ({p_1}/{p_2})^{(m-1)i}
 \right), p_1^{km-k+1}p_2^{k+1}, p_1^{km-k+1}p_2^{k+1}
\right).$$

Denote  $R_k=\frac{W_{km+2}(z_{km+2,1})+W_{km+2}(z_{km+2,2})}{W_{km+2}(z_{km+2,2})+W_{km+2}(z_{km+2,3})}$, then after simplification, we have
\begin{equation}\label{eq:Rk}
R_{k}=\frac{1}
{2({p_1}/{p_2})^{k(m-1)-1}}+
\frac{\Sigma_{i=0}^{k-1} ({p_1}/{p_2})^{(m-1)i}}
{2({p_1}/{p_2})^{k(m-1)-m}}
+\frac{1}{2}.
\end{equation}
 If $p_1=p_2=1/2$,  then  the second term tends to $\infty$ as $k\to \infty$;  if $p_1<p_2$, then  the first term   tends to $+\infty$ as $k\to\infty$. So we always have $\lim_{k\to +\infty}R_k=+\infty.$
Therefore, by Lemma \ref{lem:balance-1} , $\mu_\mathbf{p}$ is not doubling.
\end{proof}

\section{\textbf{An alternative proof of  Proposition \ref{prop:Yung}}}
\indent  If $m=2$, then $\beta$ is the golden  number; in this case,
Yung\cite{Yung07}
  proved that $\mu_\mathbf{p}$ is doubling on $[0,1]$ if and only if $p_1=p_2=\frac{1}{2}$.
In this section,  we  give an alternative  proof of this result.

First, we have that $\mathcal{D}=\{d_0=1, d_1=1/\beta, d_2=1/\beta^2\}$.    But here after we use the
alphabet $\{a,b,D\}$ instead of $\{d_1,d_2, d_0\}$ for simplicity.
In other words, let $(x_0, x_1)$ be two consecutive points in $X_n$, we define $t(x_0)=a$ if  $x_1-x_0=d_1/\beta^n$
and $t(x_0)=b$ if $x_1-x_0=d_2/\beta^n$. Let $\sigma$ be the substitution
$$\sigma: D\mapsto aD, a\mapsto ab, b\mapsto a.$$
Then the label sequence associated with $X_n$ is $\sigma^n(D)$.

Next, we change the label of the first basic interval in $X_n$ from $a$ to $a_0$,  and
we modify the substitution to
$$
\sigma: \left \{
\begin{array}{l}
a_0\mapsto a_0b\\
a\mapsto ab\\
b\mapsto a\\
D\mapsto aD.
\end{array}
\right .
$$
Then the label sequence of $X_n$ is $\sigma^{n-1}(a_0D)$.
For example, the label sequence of $X_2$
is $a_0baD$.  Let $(x_0, x_1, x_2)$ be a triple in $X_n$. Then
  $$t(x_0)t (x_1)t(x_2)\in \{a_0ba, aba, baa, aab, bab, baD\}.$$

 \subsection{Offsprings and paths}
Now we define offsprings of a triple in $X_n$  as following.

 Let $(x_0, x_1, x_2)$ be a triple in $X_n$.
Let $(z_0=x_0, z_1,z_2,z_3,z_4)$ be five consecutive
 points in $X_{n+1}$. We define the set of offsprings of  $(x_0, x_1, x_2)$ to be
 $$
 \left \{
 \begin{array}{ll}
 \{(z_0, z_1, z_2), (z_1, z_2, z_3), (z_2,z_3, z_4)\}, &\text{ if } t(x_0)t (x_1)t(x_2)=a_0ba,\\
 \{(z_1, z_2, z_3), (z_2,z_3, z_4)\}, &\text{ if } t(x_0)t (x_1)t(x_2)\in\{aba,aab, baD\},\\
  \{  (z_1, z_2, z_3) )\}, &\text{ if } t(x_0)t (x_1)t(x_2)\in\{baa,bab\}.
 \end{array}
 \right .
 $$
It is easy to verify that

  (i) If a triple has more than one offsprings, then the label of the offsprings are all different;

  (ii) Every triple in $X_{n+1}$  has one and only one ancestor, which is a   triple of $X_n$.

  For simplicity, we denote
  $$S_0=a_0ba, S_1=aba, \ S_2=baa, \ S_3=baD, \
  H_1=aab,   \ H_2=bab.$$
   Figure 2 illustrates the offspring relationship, where $S\mapsto T$ means that $T$ is an offspring of $S$.
\begin{figure}[H]
  \centering
    \includegraphics[width=8 cm]{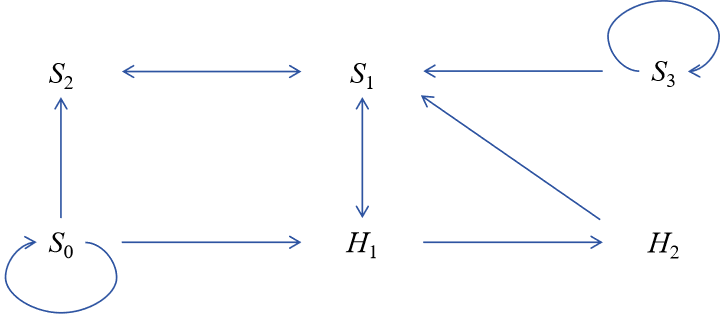}
    \caption{Offsprings of triples.}
    \label{fig:graph}
\end{figure}

Let $(Z_k)_{k=2}^n$ be a sequence such that $Z_k$ is a triple in $X_k$. We call
it a \emph{path} to $Z_n$ if $Z_{k+1}$ is an offspring of $Z_k$  for $2\leq k\leq n-1$.

\subsection{Transition matrix}
Suppose $(x_0,x_1,x_2)$ is a triple in $X_n$ with label $S$,
let $(y_0, y_1, y_2)$ be an offspring of $(x_0,x_1,x_2)$ with label $T$.
Then there is a $3\times 3$ matrix $M_{ST}$ such that
\begin{equation}
    \begin{bmatrix}
        W_{n+1}(y_0)  \\
        W_{m+2}(y_1)  \\
        W_{m+2}(y_2)
    \end{bmatrix}
    =M_{ST}
    \begin{bmatrix}
        W_{n}(x_0)  \\
        W_{n}(x_1)  \\
        W_{n}(x_2)
    \end{bmatrix},
\end{equation}
 An easy calculation shows that
\[
\begin{array}{ccc}
M_{S_0S_0}=
    \begin{bmatrix}
        p_1 & 0 & 0 \\
        p_2 & 0 & 0 \\
        0 & p_1 & 0
    \end{bmatrix},
    &
    M_{S_0S_2}=M_{S_1S_2}=
    \begin{bmatrix}
        p_2 & 0 & 0 \\
        0 & p_1 & 0 \\
        0 & p_2 & p_1
    \end{bmatrix},
    &
M_{S_0H_1}=M_{S_1H_1}=
    \begin{bmatrix}
        0 & p_1 & 0 \\
        0 & p_2 & p_1 \\
        0 & 0 & p_2
    \end{bmatrix},
\end{array}
\]
\[
\begin{array}{ccc}
M_{H_1H_2}=
    \begin{bmatrix}
        p_2 & 0 & 0 \\
        0 & p_1 & 0 \\
        0 & p_2 & 0
    \end{bmatrix},
    &
M_{H_1S_1}=
    \begin{bmatrix}
        0 & p_1 & 0 \\
        0 & p_2 & 0 \\
        0 & 0 & p_2
    \end{bmatrix};
&
M_{S_2S_1}=
    \begin{bmatrix}
        p_2 & p_1 & 0 \\
        0 & p_2 & 0 \\
        0 & 0 & p_1
    \end{bmatrix},
    \end{array}
    \]
    \[
    \begin{array}{ccc}
M_{H_2S_1}=
    \begin{bmatrix}
        p_2 & p_1 & 0 \\
        0 & p_2 & 0 \\
        0 & 0 & p_1
    \end{bmatrix};
&
M_{S_3S_1}=
    \begin{bmatrix}
        p_2 & p_1 & 0 \\
        0 & p_2 & 0 \\
        0 & 0 & p_1
    \end{bmatrix},
    &
M_{S_3S_3}=
    \begin{bmatrix}
        0 & p_2 & 0 \\
        0 & 0 & p_1 \\
        0 & 0 & p_2
    \end{bmatrix}.
\end{array}
\]

\begin{lemma}\label{lem:balanced}
If   $p_1=p_2=1/2$  and $(z_0,z_1,z_2)$ is a triple with label $H_1$ or $H_2$, then the corresponding weight triple is
 of the form $(y,y+z,z)/2$ or $(x,y,y)/2$  for some $x,y,z>0$ respectively.
\end{lemma}

 \begin{proof} Suppose $(z_0, z_1,z_2)$ is an offspring of $(x_0,x_1,x_2)$ in $X_n$; denote the weight triple
of $(x_0,x_1,x_2)$ by $(x,y,z)$. If the label of $(z_0, z_1,z_2)$ is $H_1$, then the label  of  $(x_0,x_1,x_2)$ must be
$S_0$ or $S_1$, and (notice that $M_{S_0H_1}=M_{S_1H_1}$)
$$
 \begin{bmatrix}
 W_{n+1}(z_0)  \\
 W_{n+1}(z_1)  \\
  W_{n+1}(z_2)
    \end{bmatrix}
    =
     M_{S_0H_1}
  \begin{bmatrix}
 x  \\
  y  \\
   z
    \end{bmatrix}=
    \begin{bmatrix}
    p_1 y\\
    p_2 y+p_1 z\\
    p_2 z
    \end{bmatrix}
    =  \frac{1}{2}
     \begin{bmatrix}
 y  \\
  y+z  \\
   z
    \end{bmatrix}
    .
    $$
 Similarly,   if the label of $(z_0, z_1,z_2)$ is $H_2$, then
    the label of $(x_0, x_1,x _2)$ must be $H_1$, so  the weight triple of $(z_0, z_1,z_2)$
    is   $(x, y, y)/2$.
 \end{proof}

\subsection{Verifying the doubling property}


\begin{proof}[\textbf{Proof of Proposition \ref{prop:Yung}}] We say a vector $(w_1,w_2,w_3)$ is \emph{$2$-balanced}
if $(w_1+w_2)/(w_2+w_3)$ belongs to $[1/2, 2]$. Let $M^T$ denotes the transpose of a matrix $M$.

The 'if' part: Suppose $p_1=p_2=\frac{1}{2}$.
Let $(Z_k)_{k=2}^n$ be a path to $Z_n$. We shall show that the weight triple of $Z_n$ is
$2$-balanced.  It is seen that
 \begin{equation}\label{eq:cycle}
U_k:=(M_{S_2S_1}M_{S_1S_2})^k
=\frac{1}{2^{2k}}
    \begin{bmatrix}
        1 & 1 & 0 \\
        0 & 1 & 0 \\
        0 & 1 & 1
    \end{bmatrix}^k\\
=\frac{1}{2^{2k}}
    \begin{bmatrix}
        1 & k & 0 \\
        0 & 1 & 0 \\
        0 & k & 1
    \end{bmatrix}.
\end{equation}

Denote by $u_k$ the label of $Z_k$ for $2\leq k\leq n$.
It is seen that $u_2=S_0$ or $S_3$.
Let $t$ be the maximal integer such that $u_t\not\in  \{S_1, S_2\}$.

\textit{Case 1.}  $u_t\in \{H_1, H_2, S_3\}$.

    By Figure  \ref{fig:graph},
 \begin{equation}\label{eq:S1S2}
 u_2\dots u_n=u_2\dots u_{t-1}u_t(S_1S_2)^k \text{   or  }  u_2\dots u_{t-1}u_t(S_1S_2)^k S_1.
 \end{equation}

It $u_t=H_1$,  by Lemma \ref{lem:balanced},   $W(Z_t)=(y,y+z,z)/2$ for some $y,z>0$.
  In the first scenario of \eqref{eq:S1S2},  the weight triple of $Z_n$ is
\begin{eqnarray*}
\frac{1}{2}M_{S_1S_2}\cdot U_{k-1}\cdot M_{H_1S_1}\cdot
\begin{bmatrix}
y\\
y+z\\
z
\end{bmatrix}
=\frac{1}{2^{2k+1}}
    \begin{bmatrix}
        k(y+z) \\
        y+z \\
        k(y+z)+z
    \end{bmatrix}
\end{eqnarray*}
 which is $2$-balanced;
 in the second case,  the weight triple of $Z_n$ is
\begin{eqnarray*}
\frac{1}{2}U_k\cdot M_{H_1S_1}\cdot
\begin{bmatrix}
y\\
y+z\\
z
\end{bmatrix}
=\frac{1}{2^{2k+2}}
    \begin{bmatrix}
        (k+1)(y+z) \\
        y+z \\
        k(y+z)+z
    \end{bmatrix}
\end{eqnarray*}
 which is still $2$-balanced.

 If $u_t=H_2$, then the  the weight triple of $Z_t$ is   $ (x,y,y)/2$ for some $x,y>0$; if $u_t=S_3$, then $u_2\dots u_t=(S_3)^{t-1}$ and
 \begin{eqnarray*}
 W(Z_t)=M_{S_3S_3}^{t-2}W(Z_2)=\frac{1}{2^{t-2}}
 \begin{bmatrix}
 0 & 0 & 1 \\
        0 & 0 & 1 \\
        0 & 0 & 1
 \end{bmatrix}
   \begin{bmatrix}
   1/2^2 \\
        1/2^2 \\
        1/2^2
 \end{bmatrix}
 =\frac{1}{2^t}
   \begin{bmatrix}
   1 \\
        1 \\
        1
 \end{bmatrix}.
 \end{eqnarray*}
 Hence, similar as above, one can show that
  the weight triple of $Z_n$ is   $2$-balanced.

 \textit{Case 2.}  $u_t=S_0$.
   By Figure  \ref{fig:graph},
   $$u_2\dots u_n=(S_0)^{t-1}(S_2S_1)^k \text { or } (S_0)^{t-1}(S_2S_1)^kS_2. $$
 First, we have $W(Z_t)=(1,1,1)/2^{t}$.
 Secondly, we have
 $$W(Z_n)=U_{k-1}\cdot M_{S_2S_1}\cdot M_{S_0S_2}\cdot W(Z_t)=\frac{1}{2^{2k+t}}(k+1,1,k+1),$$
 $$\text{ or } \quad \quad  W(Z_n)=M_{S_1S_2}\cdot U_{k-1}\cdot M_{S_2S_1}\cdot M_{S_0S_2}\cdot W(Z_t)=\frac{1}{2^{2k+t+1}}(k+1,1,k+2).
 $$

 In all these cases, $W(Z_n)$ are $2$-balanced.
 The 'if' part is proved.

  The 'only if' part:  Similar to Theorem \ref{thm:main}, we may assume that $p_1<p_2$ by the symmetry of the IFS. Let $(Z_k)_{k=2}^{2\ell+2}$ be the path with    $u_2\cdots u_{2\ell+2}=(S_0S_2)(S_1S_2)^{\ell-1}S_1$.
  Denote
\[
    M=M_{S_2S_1}M_{S_1S_2}=\begin{bmatrix}
        p_2^2 & p_1^2 & 0 \\
        0 & p_1p_2 & 0 \\
        0 & p_1p_2 & p_1^2
    \end{bmatrix},
\]

then\[M^\ell=
    \begin{bmatrix}
        p_2^{2\ell} & \Sigma_{i=2}^{\ell+1}p_1^{i}p_2^{2\ell-i} & 0 \\
        0 & p_1^{\ell}p_2^{\ell} & 0 \\
        0 & \Sigma_{i=\ell}^{2\ell-1}p_1^{i}p_2^{2\ell-i} & p_1^{2\ell}
    \end{bmatrix}.
\]
 Notice that $W(Z_2)=(p_1^2,p_1p_2,p_1p_2)$ and
$$W(Z_{2\ell +2})=(M_{S_2S_1}M_{S_1S_2})^{\ell-1}M_{S_2S_1}M_{S_0S_2}W(Z_2)=M^\ell W(Z_2).$$
Therefore,  the weight triple of $Z_{2\ell+2}$ is
$$W(Z_{2\ell +2})=(p_1^2p_2^{2\ell}+\Sigma_{i=3}^{\ell+2}p_1^{i}p_2^{2(\ell+1)-i}, p_1^{\ell+1}p_2^{\ell+1}, \Sigma_{i=\ell+1}^{2\ell+1}p_1^{i}p_2^{2(\ell+1)-i}).$$
It follows that

$$R_\ell:=\frac{p_1^2p_2^{2\ell}+\Sigma_{i=3}^{\ell+2}p_1^{i}p_2^{2(\ell+1)-i}+ p_1^{\ell+1}p_2^{\ell+1}}
{p_1^{\ell+1}p_2^{\ell+1}+ \Sigma_{i=\ell+1}^{2\ell+1}p_1^{i}p_2^{2(\ell+1)-i}}
\geq \frac{p_1^2p_2^{2\ell}}{(l+2)p_1^{\ell+1}p_2^{\ell+1}}=\frac{1}{\ell+2}(\frac{p_2}{p_1})^{\ell-1}.
$$
So
 $\lim_{\ell\to +\infty}R_\ell=+\infty.$
This means that $\mu_\mathbf{p}$ is not doubling on $[0,1]$ when $p_1<p_2$.  The 'only if' part is proved.
\end{proof}


\begin{thebibliography}{99}
\addcontentsline{toc}{chapter}{Bibliography}
\bibitem{Bertin92} M. J. Bertin, A. Decomps-Guilloux, M.Grandet-Hugot, M.Pathiaux-Delefosse and J.P.Schreiber, {\it Pisot and Salem Numbers}, Birkh\"auser, 1992.

\bibitem{Beurling_1956} A. Beurling and L. Ahlfors, {\it The boundary correspondence under quasiconformal mappings,} Acta Math., \textbf{96} (1956), 125-142.
\bibitem{Erdos_1939} P. Erd\"{o}s, {\it On a family of symmetric Bernoulli convolutions}, Amer. J. Math.,
  \textbf{61}(1939), 974-975.
\bibitem{Garsia_1962} A. M. Garsia, {\it Arithmetic properties of Bernoulli convolutions}, Trans. Amer. Math. Soc.,
  \textbf{102}(1962), 409-432.
\bibitem{Hutchinson_1981} J. E. Hutchinson, {\it Fractals and self-similarity,} Indiana Univ. Math. J., \textbf{30} (1981), 713-747.
\bibitem{Mauldin_1996} R. D. Mauldin and M. Urba\'nski, {\it Dimensions and measures in infinite iterated function systems,} Proc. London Math. Soc., \textbf{73} (1996), 105-154.
\bibitem{Olsen_1995} L. Olsen, {\it A multifractal formalism,} Adv. Math., \textbf{116} (1995), 82-196.
\bibitem{Solomyak_1995} B. Solomyak, {\it on the ransom series $\sum \pm \lambda^i$ (an Erd\"{o}s problem)}, Ann. of Math.,
  \textbf{142}(1995), 611-625.
\bibitem{Yang_2024} Y. M. Yang, Q. H. Yuan, and Y. Zhang, {\it Doubling properties of self-similar measures and Bernoulli measures on self-affine Sierpi¨½ski sponges,} Indiana Univ. Math. J., \textbf{73} (2024), 475-492.
\bibitem{Yung07} P. L. Yung, {\it Doubling properties of self-similar measures}, Indiana Univ. Math. J., \textbf{56} (2007), 965-990.


\end{thebibliography}
\end{document}